\documentclass[twoside,11pt]{article}
\usepackage{graphicx,amsmath,amsthm,amssymb,latexsym,amsfonts,color}
\usepackage[bookmarksnumbered, colorlinks, plainpages]{hyperref}
\def\thefootnote{}
\newtheorem{thm}{\bf Theorem}
\newtheorem{lem}{\bf Lemma}

\newtheorem{example}{\bf Example}

\newtheorem{remark}{\bf Remark}

\newtheorem{rem}{Remark}
\newtheorem{cor}{Corollary}

\def\h{\hspace{-0.2cm}}
\def\argmin{{\rm{arg}}\hspace{-0.05cm}\min}
\begin{document}
\title{Two-parameter TSCSP method for solving complex symmetric  system of linear equations}
\author{{ Davod Khojasteh Salkuyeh and Tahereh Salimi Siahkolaei}\\[2mm]
\textit{{\small Faculty of Mathematical Sciences, University of Guilan, Rasht, Iran}} \\
\textit{{\small E-mails: khojasteh@guilan.ac.ir, salimi-tahereh@phd.guilan.ac.ir}\textit{}}}
\date{}
\maketitle
\noindent{\bf Abstract.} We introduce a two-parameter version of the two-step scale-splitting iteration method, called TTSCSP, for solving a broad class of complex symmetric system of linear equations. We present some conditions for the convergence of the method. An upper bound for the spectral radius of the method is presented and optimal parameters which minimize this bound are given. Inexact version of the TTSCSP iteration method (ITTSCSP) is also presented. Some numerical experiments are reported to verify  the effectiveness of the TTSCSP iteration method and the numerical results are compared with those of the TSCSP, the SCSP and the PMHSS iteration methods. Numerical comparison of the ITTSCSP method with the inexact version of TSCSP, SCSP and PMHSS are  presented. We also compare the numerical results of the BiCGSTAB method in conjunction with the TTSCSP and the ILU preconditioners.
 \\[-3mm]


\pagestyle{myheadings}\markboth{D.K. Salkuyeh and  T. Salimi Siahkolaei}{Two-parameter TSCSP method}

\thispagestyle{empty}

\section{Introduction} \label{SEC1}

Consider the system of linear equations of the form
\begin{equation}\label{Eq1}
Au=(W+iT)u=b,\end{equation}
where $W,T \in \mathbb{R }^{n \times n}$, $u=x+iy$ and $b=p+iq$, such that the vectors $x,y,p$ and $q$ are in $\mathbb{R}^{n}$ and $i=\sqrt{-1}$.
We assume that the matrices $W$ and $T$ are symmetric positive semidefinite matrices with at least one least one of them, e.g., $W$, is positive definite.
Systems of the form (\ref{Eq1}) arise in many important problems in scientific computing and engineering applications.
For example, numerical solution of the Helmholtz equation and time-dependent PDEs \cite{Bertaccini}, diffuse optical tomography \cite{Arridge},
algebraic eigenvalue problems \cite{Moro, Schmitt}, molecular scattering \cite{Poirier}, structural dynamics \cite{Feriani} and
lattice quantum chromodynamics \cite{Frommer}.

In recent years, there have been many works for solving Eq. (\ref{Eq1}), and several iterative methods have been presented in the literature.
For example, based on the Hermitian and skew-Hermitian splitting (HSS) of the matrix $A$, Bai et al. in \cite{Bai1} introduced the
Hermitian/skew-Hermitian splitting (HSS) method to solve non-Hermitian positive definite system of linear equations. Next, Bai et al.
presented a modified version of the HSS iterative method say (MHSS) \cite{Bai2} to solve systems of the form (\ref{Eq1}).
Then, a preconditioned version of the MHSS iteration method, called PMHSS, was presented by Bai et al. in \cite{Bai3}.

Let
\begin{equation}\label{Eq2}
A=H+S,
\end{equation}
be the Hermitian/Skew-Hermitian (HS) splitting of the matrix $A$, where
\[
H=\frac{1}{2}(A+A^H)=W, \quad S=\frac{1}{2}(A-A^H)=iT,
\]
with $A^H$ being the conjugate transpose of $A$. Let also $V\in\Bbb{R}^{n\times n}$ be a symmetric positive definite.  Then, the PMHSS iteration method
can be described as follows.

\bigskip

\noindent {\bf The PMHSS method}: \verb"Let" $u^{(0)} \in  {\mathbb{C }^{n}}$ \verb"be an initial guess". \verb"For" $k=0,1,2,\ldots$, until $\{u^{(k)}\}$ \verb"converges, compute" ${u^{(k+1)}}$  \verb"according to the following sequence":
\begin{equation}\label{Eq4}
\begin{cases}
(\alpha V+W){u^{(k+\frac{1}{2})}}=(\alpha V-iT){u^{(k)}}+b, \\
(\alpha V+T){u^{(k+1)}}=(\alpha V+iW){u^{(k+\frac{1}{2})}}-ib,
\end{cases}
\end{equation}
\verb"where" $\alpha$ \verb"is a given positive constant."

\bigskip

When the matrix  $V$ is equal to the identity matrix, then the PMHSS iteration method reduces to MHSS. In \cite{Bai3},
it has been proved that the PMHSS iteration converges unconditionally to the unique solution of the complex symmetric
system (\ref{Eq1}) for any initial guess. Numerical implementation presented in \cite{Bai3} show that a Krylov subspace
iteration method such as generalized minimal residual (GMRES) \cite{saad} in conjunction with the resulting PMHSS preconditioner  is very efficient to solve
the system (\ref{Eq1}). In particular, both the PMHSS iteration method and the MHSS-preconditioned GMRES show meshsize-independent
and parameter-insensitive convergence behaviour (see \cite{Bai3}).

In \cite{Axelsson}, Axelsson and Kucherov showed that it is possible to avoid complex arithmetic by rewriting Eq. \eqref{Eq1} to several real-valued forms. Among them, we consider the follwing real-form
\begin{equation}\label{Eq5}
\mathcal{A} u=
 \begin{bmatrix}
W & -T \\
T & W
\end{bmatrix}
 \begin{bmatrix}
x \\
y
 \end{bmatrix}
 =
\begin{bmatrix}
p \\
q
\end{bmatrix}.
\end{equation}
\noindent Under our hypotheses, it can be easily proved that the matrix $A$ is nonsingular. Benzi and  Bertaccini in \cite{Benzi-Berta} investigated several block preconditioners for real equivalent formulations of
complex linear systems when the coefficient matrix $A$ is complex symmetric.
Bai et al. in \cite{BBCW} presented a preconditioned modified Hermitian and  skew-Hermitian splitting iteration method for solving and preconditioning of the system \eqref{Eq5}. Bai et al. in \cite{Bai4}, applied the generalized successive overrelaxation (GSOR)  method for augmented linear systems. Recently, using the idea of \cite{Bai4},  Salkuyeh et al. in \cite{Salkuyeh1} solved the system (\ref{Eq5}) by the generalized successive overrelaxation (GSOR) iterative method.
This method can be written as follows.

\bigskip

\noindent {\bf The GSOR iteration method}: \verb"Let" $(x^{(0)}; y^{(0)}) \in  {\mathbb{R }^{n}}$ \verb"be an initial guess. For" $k=0,1,2,\ldots$, \verb"until" $\{(x^{(k)};y^{(k)})\}$ \verb"converges, compute" ${(x^{(k+1)};y^{(k+1)})}$  \verb"according to the" \\ \verb"following sequence"
\begin{equation}\label{Eq6}
\begin{cases}
W x^{(k+1)}=(1-\alpha)Wx^{(k)}+\alpha T y^{(k)}+\alpha p, \\
Wy^{(k+1)}=-\alpha T x^{(k+1)}+(1-\alpha)W y^{(k)}+ \alpha q,
\end{cases}
\end{equation}
\verb"where" $\alpha$ \verb"is a given positive constant".

\bigskip

In \cite{Salkuyeh1}, it has been shown that if $W$ and $T$ are symmetric positive definite and symmetric, respectively,
then the GSOR method is convergent.

Recently, using the matrix splitting
\[
A=\frac{1}{\alpha-i}\left[(\alpha W+T)-i(W-\alpha T)\right],
\]
Hezari et al. in \cite{hezari1} presented the Scale-Splitting (SCSP) iteration method for solving  (\ref{Eq1}) which can be described as follows.

\bigskip

\noindent {\bf The SCSP iteration method}: \verb"Let" $u^{(0)} \in  {\mathbb{C }^{n}}$ \verb"be an initial guess. For" $k=0,1,2,\ldots$, \verb"until" $\{u^{(k)}\}$ \verb"converges, compute" ${u^{(k+1)}}$  \verb"according to the" \verb"following sequence"

\begin{equation}\label{Eq06}
(\alpha W+T) u^{(k+1)}=i(W-\alpha T)u^{(k)}+ (\alpha-i)b,
\end{equation}
\noindent where $\alpha$ is a given positive constant.
\bigskip

It can be seen that the SCSP iteration method is equivalent to the matrix splitting iteration method induced by the splitting defined through the additive block diagonal (ABD) preconditioner introduced and discussed by Bai et al. in \cite{Bai5}.   At each iteration of the SCSP iteration method, it is required to solve a linear system with coefficient matrix $\alpha W+T$.  In \cite{hezari1} it was proved
that if $W$ and $T$ are symmetric positive semidefinite matrices satisfying ${\rm null}(W)\cap {\rm null}(T)=\{0\}$, then the SCSP iteration method
is convergent provided that
\[
  \left\{
    \begin{array}{ll}
      \displaystyle{\frac{1-\mu_{\min}}{1+\mu_{\min}}<\alpha <\frac{1+\mu_{\max}}{\mu_{\max}-1}}, & \quad \hbox{for $\mu_{\max}>1$}, \\
\\
      \displaystyle{\frac{1-\mu_{\min}}{1+\mu_{\min}}<\alpha}, & \quad \hbox{for $\mu_{\max}\leq1$,}
    \end{array}
  \right.
  \]
\noindent where $\mu_{\min}$ and $\mu_{\max}$ are the smallest and largest generalized eigenvalues of the matrix pair $(W,T)$, respectively. Recently, using the idea of the SCSP iteration method, Salkuyeh in \cite{Salkuyeh2} presented a two-step Scale-Splitting (TSCSP) for solving Eq. (\ref{Eq1}) which
is algorithmically described in the following form (see also \cite{ZZheng}).

\bigskip

\noindent {\bf The TSCSP iteration method}: \verb"Let" $u^{(0)} \in  {\mathbb{C }^{n}}$ \verb"be an initial guess. For" $k=0,1,2,\ldots$, \verb"until" $\{u^{(k)}\}$ \verb"converges, compute" ${u^{(k+1)}}$  \verb"according to the" \verb"following sequence"

\begin{equation}\label{Eq6}
\begin{cases}
(\alpha W+T) u^{(k+\frac{1}{2})}=i(W-\alpha T)u^{(k)}+ (\alpha-i)b,\\
( W+\alpha T) u^{(k+1)}=i(\alpha W-T)u^{(k+\frac{1}{2})}+ (1-\alpha i)b,
\end{cases}
\end{equation}
\noindent where $\alpha>0$.

\bigskip

Theoretical analysis in \cite{Salkuyeh2} indicate that if the matrices $W$ and $T$ are symmetric positive definite, then the
TSCSP iteration method  unconditionally converges. Numerical results presented in \cite{Salkuyeh2} show that the TSCSP iteration method
outperforms the PMHSS, the GSOR, the SCSP iteration methods. When $\alpha=1$, the TSCSP iteration method reduces to the ABD itetation method \cite{Bai5}. In this paper we present a two parameter TSCSP iteration method to solve the system (\ref{Eq1}) and analyze its convergence properties.

In the PMHSS, the GSOR, the SCSP and TSCSP iteration methods it is required to solve some subsystems with symmetric positive definite coefficient matrices.
These systems  can be solved exactly by using the Cholesky factorization of the coefficient matrices or inexactly by the conjugate gradient (CG) iteration method or its preconditioned version (PCG).

The remainder of the paper is organized as follows. In Section \ref{Sec2}, the TTSCSP iteration method is established and the convergence of the method is discussed. Inexact version of the TTSCSP method is studied in Section \ref{Sec3}.  Section \ref{Sec4} is devoted to some numerical experiments to show the effectiveness of TTSCSP. Finally, some concluding remarks are given  in Section \ref{Sec5}.
\section{The TTSCSP iteration method} \label{Sec2}

In this section, we derive a new version of the TSCSP iteration method that was initially proposed in \cite{Salkuyeh2}. The new method will be referred to as two-parameter TSCSP (TTSCSP) iteration method or, in brief, the TTSCSP iteration method.
To this end, let $\alpha>0$. By multiplying $(\alpha-i)$ through both sides of the complex system (\ref{Eq1}) we obtain the following equivalent system

\begin{equation}\label{Eq10}
(\alpha-i)Au=(\alpha-i)b,
\end{equation}
\noindent where $i=\sqrt{-1}$. The latter equation results in the following system of fixed-point equation

\begin{equation}\label{Eq11}
(\alpha W+T)u=i(W-\alpha T)u+(\alpha -i)b.
\end{equation}
Next, we multiply both sides of Eq. \eqref{Eq1} by $(1-\beta i)$ with $\beta>0$ to obtain the equivalent system
\begin{equation}\label{Eq12}
(1-\beta i)Au=(1-\beta i)b.
\end{equation}
\noindent It can be alternatively rewritten as the following system of fixed-point equations
\begin{equation}\label{Eq13}
(W+\beta T)u=i(\beta W-T)u+(1-\beta i)b.
\end{equation}
Now, by alternately iterating between the two systems of fixed-point equations (\ref{Eq11}) and (\ref{Eq13}), we can establish the following TTSCSP iteration method for solving the complex symmetric linear system (\ref{Eq1}).

\bigskip

\noindent {\bf The TTSCSP iteration method}: \verb"Let" $u^{(0)} \in  {\mathbb{C }^{n}}$ \verb"be an initial guess. For" $k=0,1,2,\ldots$, \verb"until" $\{u^{(k)}\}$ \verb"converges, compute" ${u^{(k+1)}}$  \verb"according to the" \verb"following sequence"
\begin{equation}\label{it}
\begin{cases}
(\alpha W+T) u^{(k+\frac{1}{2})}=i(W-\alpha T)u^{(k)}+ (\alpha-i)b,\\
( W+\beta T) u^{(k+1)}=i(\beta W-T)u^{(k+\frac{1}{2})}+ (1-\beta i)b,
\end{cases}
\end{equation}
 where $\alpha$ and $\beta$  are positive numbers.

\bigskip
It is mentioned that when $\alpha=\beta$, the TTSCSP method reduces to the TSCSP method.
The two subsystems of each iterate of the TTSCSP method require  to solve the systems with coefficient matrices $\alpha W+T$ and $W+\beta T$.  If $W$ and $T$ are  symmetric positive definite and  symmetric positive semidefinite, respectively, then coefficient matrix of two subsystems, $\alpha W+T$ and $W+\beta T$, are symmetric positive definite. Therefore, the two subsystems of iteration method can be exactly solved by Cholesky factorization. This is very costly and impractical for large real problems. To improve the computing efficiency of TTSCSP, we can inexactly solve the involving subsystems by CG or PCG.

The TTSCSP iteration method can be reformulated as the form
\begin{equation}\label{Eq14}
u^{(k+1)}=\mathcal{G}_{\alpha , \beta} u^{(k)} +\mathcal{C}_{\alpha , \beta} ,
\end{equation}
\noindent where
\[
\mathcal{G}_{\alpha , \beta}=(W+\beta T)^{-1}(T-\beta W)(\alpha W+T)^{-1}(W-\alpha T),
\]
and
\[
\mathcal{C}_{\alpha , \beta}=(\alpha+\beta)(W+\beta T)^{-1}(W-iT)(\alpha W+T)^{-1}b.
\]
Setting
\begin{eqnarray*}
M&\h=\h&\frac{1}{\alpha+\beta}(\alpha W+T)(W-i T)^{-1}(W+\beta T),\\
N&\h=\h&\frac{1}{\alpha+\beta}(T-\beta W)(W-i T)^{-1}(W-\alpha T),
\end{eqnarray*}
we have $A=M-N$ and $\mathcal{G}_{\alpha , \beta}={M}^{-1}N$. Therefore, the matrix
\begin{equation}\label{precondition}
Q= (\alpha W+T)(W-i T)^{-1}(W+\beta T),
\end{equation}
can be used as a preconditioner (TTSCSP preconditioner) for the system (\ref{Eq1}).

In the sequel, we prove that under suitable conditions, the TTSCSP iteration method converges to the unique solution of system (\ref{Eq1}). To establish the convergence of the TTSCSP iteration method, the following theorem is presented.

 \begin{thm}\label{Thm1}
Let $W\in  \mathbb{R }^{n \times n}$ be symmetric positive definite and $T\in \mathbb{R}^{n \times n}$ be symmetric positive semidefinite and
\[
0\leqslant \mu_1 \leqslant \cdots \leqslant \mu_r  <1 \leqslant \mu_{r+1}\leqslant \cdots \leqslant \mu_n,
\]
be  the eigenvalues of $S=W^{-\frac{1}{2}}TW^{-\frac{1}{2}}$.
Then, the TTSCSP iteration method is convergent, i.e., $\rho\left(\mathcal{G}_{\alpha , \beta}\right)<1$,  if $\alpha$ and $\beta$ satisfy
\begin{equation}\label{CondThm1}
 \frac{1-\mu_1}{1+\mu_1}<\alpha<\frac{\mu_n+1}{\mu_n-1}\quad and \quad \frac{\mu_n-1}{\mu_n+1}<\beta<\frac{1+\mu_1}{1-\mu_1}.
\end{equation}
\begin{proof}
Let
\[
{\hat{\mathcal{G}}}_{\alpha , \beta}=(I+\beta S)^{-1}(S-\beta I)(\alpha I+S)^{-1}(I-\alpha S),
\]
where $S=W^{-\frac{1}{2}}TW^{-\frac{1}{2}}$. It is easy to see that ${\mathcal{G}}_{\alpha , \beta}=W^{-\frac{1}{2}} {\hat{\mathcal{G}}}_{\alpha , \beta} W^{\frac{1}{2}}$. Hence, the matrices $\mathcal{G}_{\alpha , \beta}$ and ${\hat{\mathcal{G}}}_{\alpha , \beta}$ are similar and their eigenvalues are the same. Since the matrices $W$ and $T$ are symmetric positive definite and symmetric positive semidefinite, respectively, then the eigenvalues of $S$ are nonnegative. Therefore, we have
\begin{eqnarray*}
\rho(\mathcal{G}_{\alpha , \beta})&\h=\h&\rho({\hat{\mathcal{G}}}_{\alpha , \beta})\\
                                  &\h=\h& \rho\left((I+\beta S)^{-1}(S-\beta I)(\alpha I+S)^{-1}(I-\alpha S)\right)\\
                                  &\h=\h&\displaystyle \max_{\mu_{j}\in \sigma(S)} \left| \lambda(\alpha,\beta,\mu_j) \right|.
\end{eqnarray*}
where
\begin{equation}\label{Lambda1}
\lambda(\alpha,\beta,\mu_j)=\displaystyle  \frac{(\mu_{j}-\beta)(1-\alpha \mu_{j})}{(1+\beta \mu_{j})(\alpha + \mu_{j})}.
\end{equation}

\noindent Then, we have
\[
|\lambda(\alpha,\beta,\mu_j)|=\left| \frac{\mu_{j}-\beta}{1+\beta \mu_{j}} \right|
\left| \frac{1-\alpha \mu_{j}}{\alpha + \mu_{j}} \right|,
\]
and to get $|\lambda(\alpha,\beta,\mu_j)|<1$, it is enough to have
\begin{equation}\label{Ineq0}
\left| \frac{\mu_{j}-\beta}{1+\beta \mu_{j}} \right|<1 \quad  {\rm and} \quad
\left| \frac{1-\alpha \mu_{j}}{\alpha + \mu_{j}} \right|<1.
\end{equation}
The left inequality in (\ref{Ineq0}) is equivalent to the following inequalities
\begin{eqnarray}
\beta (1-\mu_j)&\h<\h& 1+\mu_j, \label{EQSS1}\\
\beta(\mu_j+1) &\h>\h& \mu_j-1.   \label{EQSS2}
\end{eqnarray}
For $\mu_j\geqslant 1$, these inequalities hold true if and only if
\begin{equation}\label{Ineq1}
\beta>\frac{\mu_j-1}{\mu_j+1},
\end{equation}
and for $0<\mu_j<1$, the inequalities (\ref{EQSS1}) and (\ref{EQSS2}) hold if and only if
\begin{equation}\label{Ineq2}
\beta<\frac{1+\mu_j}{1-\mu_j}.
\end{equation}
Hence, from Eqs. (\ref{Ineq1}) and (\ref{Ineq2}) it is enough to set
\[
\frac{\mu_n-1}{\mu_n+1}=\max_{\mu_j\geqslant 1}\frac{\mu_j-1}{\mu_j+1}<\beta<\min_{\mu_j<1}\frac{1+\mu_j}{1-\mu_j}=\frac{1+\mu_1}{1-\mu_1}.
\]
In a similar manner the condition for the parameter $\alpha$ can be obtained. It should be mentioned that if $\mu_1=0$ then form (\ref{CondThm1}) it can be deduced that $\beta<\alpha$. Therefore, if $\mu_j=0$ for some $j$, then
\[
|\lambda(\alpha,\beta,\mu_j)|=|\lambda(\alpha,\beta,0)|=|\frac{\beta}{\alpha}|<1.
\]
\end{proof}
\end{thm}
\begin{rem}\label{1}
 If $0\leqslant \mu_1 \leqslant \cdots \leqslant \mu_n  \leqslant 1$, then from Eqs. (\ref{EQSS1}) and (\ref{EQSS2}) the sufficient conditions for the convergence of the TTSCSP iteration method reduce to
\begin{equation*}
 \alpha>\frac{1-\mu_1}{1+\mu_1} \quad \textrm{and} \quad \beta<\frac{1+\mu_1}{1-\mu_1},
\end{equation*}
and if $1\leqslant \mu_1 \leqslant \cdots \leqslant \mu_n $, then again using (\ref{EQSS1}) and (\ref{EQSS2}) the sufficient convergence conditions reduce to
\begin{equation*}
\alpha<\frac{\mu_n+1}{\mu_n-1}\quad \textrm{and} \quad \beta>\frac{\mu_n-1}{\mu_n+1}.
\end{equation*}
\end{rem}
In general it is difficult to find the optimal values of the parameters $\alpha$ and $\beta$. In the sequel, we obtain the parameters which minimize the upper bound $\rho(\mathcal{G}_{\alpha , \beta})$. To do so, it follows from  Eq. (\ref{Lambda1}) that
\begin{eqnarray*}
\rho(\mathcal{G}_{\alpha , \beta}) &\h=\h& \max_{\mu_{j}\in \sigma(S)}|\lambda(\alpha,\beta,\mu_j)|=\displaystyle\max_{\mu_{j}\in \sigma(S)}\left|  \frac{(\mu_{j}-\beta)(1-\alpha \mu_{j})}{(1+\beta \mu_{j})(\alpha + \mu_{j})}\right|\\
&\h\leqslant\h& \max_{\mu_{j}\in \sigma(S)} \left | \frac{\mu_{j}-\beta}{1+\beta \mu_{j}} \right| .  \max_{\mu_{j}\in \sigma(S)} \left| \frac{1-\alpha \mu_{j}}{\alpha +\mu_{j}}  \right|=:\sigma {(\alpha , \beta)}.
\end{eqnarray*}
The next theorem presents the parameters $\alpha$ and $\beta$ which minimize $\sigma {(\alpha , \beta)}$.

\begin{thm}\label{Thm2}
Let all the assumptions of Theorem \ref{Thm1} hold and
\[
 (\alpha^{*},\beta^{*})= \argmin_{\hspace{-0.3cm}\alpha,\beta>0}  \sigma {(\alpha , \beta)}.
\]
 Then
\[
\alpha^{*} =\frac{\gamma+\sqrt{\gamma^2+\eta^2}}{\eta} \quad {\rm and} \quad \beta^{*}=\frac{1}{\alpha^*},
\]
where $\eta=\mu_1+\mu_n$ and $\gamma=1-\mu_1\mu_n$.
\begin{proof}
Let
\[
f_{\mu}(\alpha)=\frac{1-\alpha \mu}{\alpha +\mu},  \quad  {\rm and}  \quad   g_{\mu}(\beta)=\frac{\mu-\beta}{1+\beta \mu}.
\]
Obviously, we have
\begin{eqnarray*}
(\alpha^{*},\beta^{*}) &\h=\h& \argmin_{\hspace{-0.2cm}\alpha,\beta>0} \left\{\max_{\mu \in \sigma(S)}|f_{\mu}(\alpha)|.\max_{\mu \in \sigma(S)} |g_{\mu}(\beta)|\right\}\\
&\h=\h& \left(\argmin_{\hspace{-0.2cm}\alpha>0} \max_{\mu \in \sigma(S)}|f_{\mu}(\alpha)|,\argmin_{\hspace{-0.2cm}\beta>0} \max_{\mu \in \sigma(S)} |g_{\mu}(\beta)|\right).
\end{eqnarray*}
Therefore, we can independently obtain the values of $\alpha^*$ and $\beta^*$ via
\[
\alpha^{*}= \argmin_{\hspace{-0.3cm}\alpha>0} \max_{\mu \in \sigma(S)} | f_{\mu}(\alpha) | \hspace{0.3cm}  \textrm{and}  \hspace{0.3cm} \beta^{*}= \argmin_{\hspace{-0.3cm}\beta>0} \hspace{0.2cm} \max_{\mu \in \sigma(S)} |g_{\mu}(\beta) |.
\]
To compute the values of $\alpha^*$ and $\beta^*$, we first study some properties of the function $f_{\mu}(\alpha)$. This function passes through the points $(0,1/\mu)$ and $(1/\mu,0)$ and has two asymptotes $\alpha=-\mu$ and $y=-\mu$. We have
\[
\frac{d}{d\alpha} f_{\mu}(\alpha) = -\frac{1+{\mu}^2}{(\alpha +\mu)^2}<0,
\]
which shows that the function $f_{\mu}(\alpha)$ is strictly decreasing. Figure  1 displays the function $ |f_{\mu}(\alpha) |$ for $\mu=\mu_1,\mu_2,\mu_3$, where $\mu_{1}<\mu_2<\mu_3$. As seen the optimal values of $\alpha$ are obtained by intersecting the functions $|f_{\mu_1}(\alpha)|$ and $|f_{\mu_3}(\alpha) |$. Therefore, in general $\alpha^{*}$ satisfies the relation
 \[
 \frac{1-\alpha^* \mu_1}{\alpha^*+\mu_1}=-\frac{1-\alpha^* \mu_n}{\alpha^*+\mu_n},
 \]
which gives the following two values for $\alpha^*$,
\begin{eqnarray*}
 \alpha^*_1 &=&\frac{1-\mu_{1}\mu_{n}+\sqrt{(1-\mu_1\mu_n)^2+(\mu_1+\mu_n)^2}}{\mu_{1} + \mu_{n}}, \\
  \alpha^*_2 &=&\frac{1-\mu_{1}\mu_{n}-\sqrt{(1-\mu_1\mu_n)^2+(\mu_1+\mu_n)^2}}{\mu_{1} + \mu_{n}}
\end{eqnarray*}
Since, $\alpha^*_2$  is not  positive we deduce that $\alpha^*=\alpha_1^*$. In a similar way, the optimum value of the parameter $\beta$ can be found.
\begin{figure}\label{Fig1}
\centering
\includegraphics[height=8cm,width=12cm]{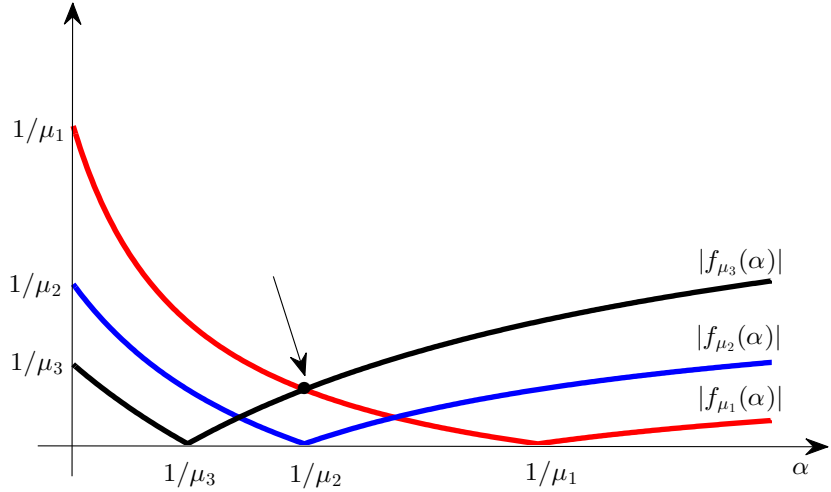}
\caption{Graph of $f_{\mu}(\alpha)$ for $\mu=\mu_1,\mu_2,\mu_3$ where $\mu_1<\mu_2<\mu_3$.}
\end{figure}
\end{proof}
\end{thm}


Other sufficient conditions for the convergence of the TTSCSP iteration are given in the next theorem.

 \begin{thm}\label{Thm3}
Let $W\in  \mathbb{R }^{n \times n}$ be symmetric positive definite and $T\in \mathbb{R}^{n \times n}$ be symmetric positive semidefinite and
$ 0=\mu_1=\cdots=\mu_{s-1} < \mu_s \leqslant  \cdots \leqslant \mu_n$ be  the eigenvalues of $S=W^{-\frac{1}{2}}TW^{-\frac{1}{2}}$.
Then, the TTSCSP iteration method is convergent if $\alpha$ and $\beta$ satisfy
\begin{equation}\label{CondThm1}
0<\beta<\alpha,\quad \alpha>\frac{1}{2}(\frac{1}{\mu_s}-\mu_s)\quad and \quad \beta>\frac{1}{2}(\mu_n-\frac{1}{\mu_n})
\end{equation}
\end{thm}
\begin{proof}
Similar to Theorem \ref{Thm1} we consider the two case $\mu_j=0$ and $\mu_j\neq 0$.
If $\mu_j=0$, then $|\lambda(\alpha,\beta,\mu_j)|=\beta/\alpha$. Therefore, a necessary condition for the convergence of the TTSCSP iteration method is $\beta<\alpha$.

Now, we assume that $\mu_j\neq 0$. In this case, we have
\[
|\lambda(\alpha,\beta,\mu_j)|=\left| \frac{ (\beta-\mu_{j})(\alpha -\frac{1}{\mu_{j}}) }  {(\beta
+\frac{1}{\mu_{j}})(\alpha+\mu_{j})}\right| =   \frac{|\beta-\mu_{j}|}{\beta
+\frac{1}{\mu_{j}}}   .  \frac{|\alpha -\frac{1}{\mu_{j}}|}{\alpha+\mu_{j} }
\]
and to get $|\lambda(\alpha,\beta,\mu_j)|<1$, it is enough to have
\begin{equation}\label{31}
\frac{|\alpha -\frac{1}{\mu_{j}} |}{\alpha+\mu_{j}}<1 \quad  {\rm and} \quad \frac{|\beta-\mu_j|}{\beta
+\frac{1}{\mu_{j}}}<1.
\end{equation}
The left inequality in \eqref{31} holds if and only if
$$
\alpha>\frac{1}{2}(\frac{1}{\mu_j}-\mu_j):= h(\mu_j).
$$
Since, the function $h(t)$ for $t\geqslant 0$ is a decreasing, this inequality holds true if we choose the parameter $\alpha$ such a way that
\begin{equation}\label{33}
\alpha>\frac{1}{2}(\frac{1}{\mu_s}-\mu_s).
\end{equation}
Similarly, the right inequality in \eqref{31} holds if we choose the parameter $\beta$ from the relation
\begin{equation}\label{35}
\beta>\frac{1}{2}(\mu_n - \frac{1}{\mu_n}).
\end{equation}
Therefore proof is complete.
\end{proof}
\begin{cor}\label{Cor1}
Let both of the matrices $W$ and $T$ be symmetric positive definite. Then, the TTSCSP iteration method is convergent if $\alpha>0$ and $\beta>0$ satisfy
\[
\alpha>\frac{1}{2}(\frac{1}{\mu_1}-\mu_1)\quad and \quad \beta>\frac{1}{2}(\mu_n-\frac{1}{\mu_n}).
\]
\begin{proof}
Since, both of the matrices $W$ and $T$ are symmetric positive definite, we deduce that the matrix $S$ is symmetric positive definite and as a result we have $\mu_1>0$. Therefore, from the proof of Theorem \ref{Thm1} the desired result is obtained.
\end{proof}
\end{cor}

\section{Inexact TTSCSP}\label{Sec3}

For computing $u^{(k+1)}$ from \eqref{it}, we should solve two subsystems with the coefficient matrices $\alpha W+T$ and $W+\beta T$, which are very costly. To improve the implementation of the TTSCSP iteration method, we can employ an iteration method for solving the two subproblems. Since $\alpha W+T$ and $W+\beta T$ are positive definite, we can solve the two subsystems by CG.

In this section, we study inexact version of the TTSCSP (ITTSCSP) iteration method where the subsystems are solved inexactly by the CG method. The subsystems involving the TTSCSP iteration method are solved by the  PCG method such that the relative residual norms are less than $\epsilon_k>0$ and $\eta_k>0$, respectively.
To do so, letting $${\bar{u}}^{(k+\frac{1}{2})}={\bar{u}}^{(k)}+{\bar{z}}^{(k)},$$ and then substituting it in the first subsystem yields
\begin{equation}\label{Eq010}
(\alpha W+T) {\bar{z}}^{(k)}=(\alpha-i)r^{(k)},
\end{equation}
where $r^{(k)}=b-Au^{(k)}$. In the same way, letting $${\bar{u}}^{(k+1)}={\bar{u}}^{(k+\frac{1}{2})}+{\bar{z}}^{(k+\frac{1}{2})},$$ the second subsystem can be written as
\begin{equation}\label{Eq011}
(W+\beta T){\bar{z}}^{(k+\frac{1}{2})}=(1-\beta i)r^{(k+\frac{1}{2})},
\end{equation}
where  $r^{(k+\frac{1}{2})}=b-Au^{(k+\frac{1}{2})}$.
In the ITTSCSP algorithm, we inexactly solve systems (\ref{Eq010}) and (\ref{Eq011}) by the CG method. The resulting algorithm
is summarized as follows.

\bigskip

\indent {\bf The Inexact TTSCSP (ITTSCSP) iteration method}\vspace{-0.2cm}
\begin{enumerate}
  \item \verb"Choose an initial guess" $u^{(0)}$ \verb"and compute" $r^{(0)}=b-Au^{(0)}$ \\[-0.68cm]
  \item \verb"For" $k=0,1,2,\ldots$ \verb"until convergence, Do"\\[-0.68cm]
  \item \qquad \verb"Compute" $r^{(k)}=b-Au^{(k)}$ \verb"and set" $\bar{r}^{(k)}=(\alpha-i)r^{(k)}$\\[-0.68cm]
  \item \qquad \verb"Solve" $(\alpha W+T){\bar{z}}^{(k)}=\bar{r}^{(k)}$ \verb"by the CG method to compute"\\[-0.5cm]

   ~\hspace{0.68cm}\verb"the approximate solution" ${\bar{z}}^{k}$ \verb"satisfying" $\|\bar{r}^{(k)}-(\alpha W+T){\bar{z}}^{(k)}\|_2\leqslant \epsilon_k \|\bar{r}^{(k)}\|_2$ \\[-0.6cm]
  \item \qquad  $u^{(k+\frac{1}{2})}:=u^{(k)}+\bar{z}^{(k)}$
  \item \qquad  \verb"Compute" $r^{(k+\frac{1}{2})}=b-Au^{(k+\frac{1}{2})}$ \verb"and set" $\bar{r}^{(k+\frac{1}{2})}=(1-\beta i)r^{(k+\frac{1}{2})}$\\[-0.68cm]
  \item \qquad \verb"Solve" $(\alpha W+T){\bar{z}}^{(k+\frac{1}{2})}=\bar{r}^{(k+\frac{1}{2})}$ \verb"by the CG method to compute the"\\[-0.5cm]

   ~\hspace{0.68cm}\verb"approximate solution" ${\bar{z}}^{(k+\frac{1}{2})}$ \verb"satisfying" $\|\bar{r}^{(k+\frac{1}{2})}-(W+\beta T){\bar{z}}^{(k+\frac{1}{2})}\|_2\leqslant \eta_k \|\bar{r}^{(k+\frac{1}{2})}\|_2$ \\[-0.6cm]
   \item \qquad  $u^{(k+1)}:=u^{(k+\frac{1}{2})}+{\bar{z}}^{(k+\frac{1}{2})}$\\[-0.6cm]
   \item \verb"EndDo"
\end{enumerate}

In the sequel we discuss the convergence of the ITTSCSP method. If $\beta$ is a positive constant, then $W+\beta T$ is nonsingular. In the case, we define the vector norm $|||x|||=\| (W+\beta T)x \|_2$ for all  $x \in \mathbb{C }^{n}$  and the matrix norm $||| X |||=\| (W+\beta T) X (W+\beta T)^{-1} \|_2$ for all $X \in \mathbb{C }^{n \times n}$. Hereafter, for a nonsingular matrix $X$, let $\kappa(X)$ be the spectral condition number of $X$.
\begin{lem}\label{lemma1}(see \cite{Bai6})
Let $W\in  \mathbb{R }^{n \times n}$ be symmetric positive definite, $T\in \mathbb{R}^{n \times n}$ be symmetric positive semidefinite and $\alpha,\beta>0$.\\
(i) If $y^{(\tau_k)}$ is the $\tau_k$th approximate solution generated by the $\tau_k$th step of the CG iteration for solving the Hermitian positive definite system of linear equations $(\alpha W+T)y=b$, then
\[
\| y^{(\tau_k)}-y^* \|_2 \leqslant \sigma_{h_1}(\alpha,\tau_k) \| y^{(0)}-y^* \|_2,
\]
where $y^*=(\alpha W+T)^{-1}b$ is the exact solution, $y^{(0)}$ is an initial guess and
\begin{eqnarray*}
\sigma_{h_1}(\alpha,\tau_k)\equiv 2 \left(  \frac{\sqrt{\kappa(\alpha W+T)}-1}{\sqrt{\kappa(\alpha W+T)}+1}  \right)^{\tau_k}.
\end{eqnarray*}
(ii) If $y^{(\nu_k)}$ is the $\nu_k$th approximate solution generated by the $\nu_k$th step of the CG iteration for solving the Hermitian positive definite system of linear equations $(W+\beta T)y=b$. Then
\begin{eqnarray*}
\| y^{(\nu_k)}-y^* \|_2 \leqslant \sigma_{h_2}(\beta,\nu_k) \| y^{(0)}-y^* \|_2,
\end{eqnarray*}
where $y^*=(W+\beta T)^{-1}b$ is the exact solution, $y^{(0)}$ is an initial guess and
\begin{eqnarray*}
\sigma_{h_2}(\beta,\nu_k)\equiv 2 \left(  \frac{\sqrt{\kappa (W+\beta T)}-1}{\sqrt{\kappa(W+\beta T)}+1}  \right)^{\nu_k}.
\end{eqnarray*}
\end{lem}
\begin{lem}\label{lemma2}
Let $W\in  \mathbb{R }^{n \times n}$ be symmetric positive definite, $T\in \mathbb{R}^{n \times n}$ be symmetric positive semidefinite and $\alpha,\beta>0$. Suppose that $S=W^{-\frac{1}{2}}TW^{-\frac{1}{2}}$ and $\mu_1\leqslant \mu_2 \leqslant \cdots \leqslant \mu_n$ are the eigenvalues of the matrix $S$. Then
\begin{eqnarray*}
\bar{\sigma}&\h:=\h&\| (\beta W-T)(\alpha W+T)^{-1}(W-\alpha T)(W+\beta T)^{-1} \|_2 \\
            &\h=\h&\| W^{\frac{1}{2}} (\beta I-S)(\alpha I+S)^{-1}(I-\alpha S)(I+\beta S)^{-1} W^{-\frac{1}{2}}\|_2
                    \leqslant \sqrt{\kappa(W)}  \rho(\mathcal{G}_{\alpha , \beta}),\\
           r&\h:=\h&\| W+\beta T \|_2  \| (W+\beta T)^{-1} \|_2=\kappa(W+\beta T),\\
c_{h}(\alpha)&\h:=\h&\| (\alpha W+T)^{-1}  (W-\alpha T) \|_2= \|W^{-\frac{1}{2}} (\alpha I+S)^{-1}  (I-\alpha S) W^{\frac{1}{2}}\|_2 \hspace{2.7cm}\\
             &\h\leqslant\h &  \sqrt{\kappa(W)}  \max_{\mu_i \in \sigma(S)}   \left \vert \frac{1-\alpha \mu_i}{\alpha+\mu_i} \right \vert
                    \leqslant   \sqrt{\kappa(W)}   \max  \left\{ \frac{1-\alpha\mu_1}{\alpha+\mu_1},\frac{\alpha \mu_n -1}{\alpha+\mu_n} \right \}, \\
c_{s}(\beta)&\h:=\h&\| ( W+\beta T)^{-1}  (\beta W- T) \|_2= \| W^{-\frac{1}{2}} (I+\beta T)^{-1}  (\beta I-S) W^{\frac{1}{2}} \|_2 \\
                    &\h\leqslant\h &  \sqrt{\kappa(W)}  \max_{\mu_i \in \sigma(S)}   \left \vert \frac{1-\alpha \mu_i}{\alpha+\mu_i} \right \vert
                      \leqslant   \sqrt{\kappa(W)}  \max  \left\{ \frac{\beta-\mu_1}{1+\beta \mu_1},\frac{\mu_n -\beta}{1+\beta \mu_n} \right \}.
\end{eqnarray*}
\begin{proof}
The proof is straightforward and is omitted.
\end{proof}
\end{lem}
\begin{thm}\label{Theorem4}
Let $W\in  \mathbb{R }^{n \times n}$ be symmetric positive definite, $T\in \mathbb{R}^{n \times n}$ be symmetric positive semidefinite and $\{\tau_k\}$ and $\{\nu_k\}$ be two sequences of positive integers. If the iterative sequence $\{u^{(k)}\}$ is generated by the ITTSCSP iteration from an initial guess $u^{(0)}$, then it holds that
\begin{eqnarray*}
||| u^{(k+1)}-u^*|||  \leqslant (\bar \sigma+\epsilon(\alpha,\beta,\tau_k,\nu_k)) ||| u^{(k)}-u^* |||,
\end{eqnarray*}
where $u^*\in \mathbb{C }^{n}$ is the exact solution of the system of linear equations \eqref{1},
\begin{equation}\label{EpsEq}
\epsilon(\alpha,\beta,\tau_k,\nu_k)=\left(1+rc_h(\alpha)\right)\left(\sigma_{h_1} c_s(\beta) + \sigma_{h_1} \sigma_{h_2} (1+c_s(\beta))\right)+r \sigma_{h_2} c_{h}(\alpha)(1+c_s(\beta)) ,
\end{equation}
with $\sigma_{h_1}(\alpha,\tau_k)$ and $\sigma_{h_2}(\beta,\nu_k)$ being defined as in Lemma \ref{lemma1} and $r, c_{h}(\alpha)$ and $c_s(\beta)$ are defined in Lemma \ref{lemma2}. Therefore, if there exists a non-negative constant $\sigma^{ittscsp}(\alpha)\in [0,1)$ such that
\[
\bar{\sigma}+\epsilon(\alpha,\beta,\tau_k,\nu_k)\leqslant \sigma^{ittscsp}, \quad k=0,1,2,\ldots,
\]
then the iterative sequence ${u^{(k)}}$ converges to $u^*\in \mathbb{C }^{n}$ with a convergence factor being at most $\sigma^{ittscsp}$.
\end{thm}
\begin{proof}
The proof is similar to that of  Theorem 4.1 in \cite{Bai6} and omitted here.
\end{proof}
\begin{remark}
Assume that $\bar{\sigma}<1$. From Eq. \eqref{EpsEq} we see that for large enough values of $\tau_k$ and $\nu_k$ we have
\[
\bar{\sigma}+\epsilon(\alpha,\beta,\tau_k,\nu_k)<1,
\]
which guarantees the convergence of  the ITTSCSP iteration method. In particular, if
\[
\rho(\mathcal{G}_{\alpha , \beta})<\frac{1}{\sqrt{\kappa(W)}},
\]
the convergence of the method is guaranteed.
\end{remark}

\section{Numerical experiments}\label{Sec4}

We use three test problems from \cite{Bai2,Bai3} to illustrate the feasibility and effectiveness of the TTSCSP iteration method and its inexact version for solving the complex system \eqref{Eq1}. To do so, we compare the numerical results of the TTSCSP iteration method with those of the PMHSS,  the SCSP, the TSCSP methods. Numerical comparisons of the inexact version of these algorithms are also performed. Further, we apply the TTSCSP preconditioner to accelerate the convergence of the the BiCGSTAB \cite{BiCGSTAB} iteration method for system \eqref{1}. Hereafter, the BiCGSTAB  method with the TTSCSP preconditioner is denoted by BiCGSTAB-TTSCSP.  {To show the effectiveness of the preconditioner, numerical results of the BiCGSTAB method are compared with those of the preconditioned BiCGSTAB in conjunction with the TTSCSP preconditioner (with the optimal values of the parameters and $\alpha=\beta=1$)  and the modified incomplete LU (ILU) factorization computed via (in \textsc{Matlab} notation)}
\begin{center}
	\verb" [L,U] = ilu(A,struct('milu','row','droptol',1e-2));"\\
\end{center}
 {for solving  the system \eqref{Eq1}}. For this purpose, we use the \verb"bicgstab" command of \textsc{Matlab} with right reconditioning. In the implementation of the TTSCSP preconditioner the systems with the coefficient matrices $\alpha W +T$ and $W+\beta T$ are solved using the Cholesky factorization of these matrices.

Numerical results are compared in terms of  both the number of iterations and the CPU time  {(in seconds)} which are, respectively, denoted by ``Iter" and ``CPU" in the tables.
In the tables a  $\dag$  (resp., $\ddag$) means that the method fails to converge in 500 iterations (resp., because of memory limitation). Iter in the BiCGSTAB method may be an integer plus 0.5, indicating convergence half way through an iteration.  In all the tests, we use a zero vector as an initial guess and stopping criterion
$$
\frac{\| b-A u^{(k)} \|_{2}}{\| b\|_{2}}<10^{-6},
$$
is always used, where $u^{(k)}=x^{(k)}+iy^{(k)}$. For all the methods (exact versions), we apply the sparse Cholesky factorization incorporated with the symmetric approximate minimum degree reordering \cite{saad} for solving the subsystems. To do so, we have used the \verb"symamd.m" command of \textsc{Matlab}. In all the inexact version of the algorithms we apply the preconditioned CG (PCG) iteration method in conjunction with the modified incomplete Cholesky factorization with dropping tolerance $10^{-2}$ as the preconditioner for solving the subsystems. In the \textsc{Matlab} notation the preconditioner can be computed using the following command
\begin{center}
\verb" LC=ichol(C,struct('michol','on','type','ict','droptol',1e-2));"\\
\end{center}
where $C$ is a given symmetric positive definite matrix. For the inexact iteration methods, the stopping criterion for the PCG iteration method is ${10}^{-2}$.
All runs are implemented in \textsc{Matlab} R2014b with a Laptop with 2.40 GHz central processing unit (Intel(R) Core(TM) i7-5500), 8 GB memory and Windows 10 operating system.


\begin{table}[!ht]
\centering\caption{Numerical results of TTSCSP, TSCSP, SCSP and PMHSS for Example \ref{Ex1} with $\tau=h$.} \label{Tbl1}
\medskip
\begin{tabular}{lllllllll} \hline
Method            &  $m$   & 32    & 64   &128    &  256    & 512   & 1024   &       \\ \hline \vspace{-0.3cm} \\ \vspace{0.0cm}

TTSCSP       & $\alpha_{opt}$     & 0.33  & 0.30   & 0.30    & 0.30   & 0.30  & 0.30  \\
                   & $\beta_{opt}$      & 1.1    & 1.1   & 1.1    &1.1    & 1.1   & 1.1 \\
                   & Iter                      &   4     &  4     &  4      & 4      & 4     &  4  \\
                   & CPU                     & 0.01  & 0.02  & 0.13   & 0.60 & 3.49 & 18.34 \\[0.2cm]
TSCSP             & $\alpha_{opt}$ & 0.46    & 0.46    & 0.46    & 0.46    & 0.46    & 0.46  \\
                      & Iter                    & 7        & 7          & 7         &  7       & 7         &  7    \\
                      & CPU                   & 0.01    & 0.03     & 0.16     & 0.84    & 5.25    &  28.01  \\[0.2cm]

SCSP             & $\alpha_{opt}$    & 0.65    & 0.65    & 0.65    & 0.65    & 0.65    & 0.65  \\
                      & Iter                     & 9        & 9         &  9        &  9        & 9         &  9    \\
                      & CPU                     & 0.02   & 0.03     & 0.11    & 0.51    & 3.49    &  15.93  \\[0.2cm]


PMHSS             & $\alpha_{opt}$   & 1.36    & 1.35    & 1.05    & 1.05     & 1.05   & 1.05  \\
                       & Iter                      & 21       & 21       & 21      & 21        & 20      & 20     \\
                       & CPU                     & 0.02     & 0.05    & 0.35   & 1.84     & 11.90  &  64.19  \\ [0.2cm] \hline


 BiCGSTAB      &     Iter                  & 39      & 58 & 78.5 & 120.5  & 164   & 218.5  \\

                       & CPU                          & 0.04        & 0.07              & 0.18                &    0.92             & 7.64    & 38.44 \\[0.2cm]

 BiCGSTAB-ILU      &     Iter                  & 5.5  & 6.5 & 8.0 & 10.5  & 12.5   & 15.5  \\

                       & CPU                          & 0.03  & 0.03  & 0.06  &    0.22  & 1.54   & 7.25 \\[0.2cm]

BiCGSTAB-TTSCSP        & $\alpha_{opt}$           & 0.33      & 0.30    & 0.30    & 0.30    & 0.30    & 0.30  \\
                                    & $\beta_{opt}$             & 1.10      & 1.10    & 1.10    & 1.10    & 1.10    & 1.10  \\
                                    & Iter                              & 2          & 2         &  2        &  2       &  2        &  2    \\
                                    & CPU                             & 0.02      & 0.05     & 0.17    & 0.85    & 5.76    &  29.08 \\[0.2cm]

BiCGSTAB-TTSCSP         & Iter                              & 2.5       & 2.5       &  2.5     &  2.5    &  2.5     &  2.5    \\
   $\hspace{0.7cm}\alpha=\beta=1$        & CPU   & 0.01      & 0.02     & 0.16   & 0.78    & 5.80  &  29.00 \\[0.2cm] \hline

 \end{tabular}
\end{table}


\begin{table}[!ht]
\centering\caption{Numerical results of ITTSCSP, ITSCSP, ISCSP and IPMHSS  for Example \ref{Ex1} with $\tau=h$. \label{Tbl2}}
\bigskip
\begin{tabular}{llllllllll} \hline
Method            &  $m$   & 32    & 64   &128    &  256    & 512   & 1024   &    2048  &   \\ \hline \vspace{-0.3cm} \\ \vspace{0.0cm}

ITTSCSP      & $\alpha_{opt}$   & 0.34    & 0.34   & 0.34    & 0.34   & 0.34   & 0.34  &  0.34 \\
                   & $\beta_{opt}$    & 1.12    & 1.12   & 1.12    &1.12    & 1.12   & 1.12  &  1.12 \\
                   & Iter                     &   4      &  4       &  4        & 4        & 4        &  4     &  4  \\
                   & CPU                    & 0.04    & 0.05   & 0.09     & 0.31   &1.79    & 7.77  & 42.20 \\[0.2cm]

ITSCSP             & $\alpha_{opt}$  & 0.46   & 0.46     & 0.46    & 0.46   & 0.46    & 0.46   & 0.46  \\
                      & Iter                     & 7        & 7          & 7         &  7       & 7         &  7      &  7 \\
                      & CPU                    & 0.04   & 0.06     & 0.13     & 0.52    & 2.74    &12.88  & 67.79\\[0.2cm]

ISCSP             & $\alpha_{opt}$  & 0.65    & 0.65     & 0.65     & 0.65    & 0.65     & 0.65   & 0.65  \\
                      & Iter                   & 9        & 9          & 9         &  9        & 9         &  9       &  9 \\
                      & CPU                   & 0.04   & 0.05     & 0.09     & 0.35    &1.90     & 7.93    & 38.91 \\

IPMHSS          & $\alpha_{opt}$   & 1.36     & 1.35    & 1.05    & 1.05     & 1.05   & 1.05   & 1.07  \\
                      & Iter                     & 21       & 21       & 21      & 21        & 20       & 20      & 20\\
                      & CPU                   & 0.07     & 0.15    & 0.42   & 1.78     & 10.79   &43.72   & 211.59   \\  \hline

 \end{tabular}
\end{table}

\begin{table}[!ht]
	\centering\caption{Numerical results of TTSCSP, TSCSP, SCSP and PMHSS for Example \ref{Ex1} with $\tau=500h$}. \label{Tbl1-b}\\
	\begin{tabular}{lllllllll} \hline
		Method            &  $m$   & 32    & 64   &128    &  256    & 512   & 1024   &       \\ \hline \vspace{-0.3cm} \\ \vspace{0.0cm}
		
		TTSCSP       & $\alpha_{opt}$     & 0.37  & 0.49  & 0.58   & 0.63   & 0.65  &  0.66 \\
		& $\beta_{opt}$ & 1.00 & 1.00  & 1.00   & 1.00   & 1.00  &  1.00\\
		& Iter          &   2  &  2    &  2      & 2     & 2     &  2  \\
		& CPU           & 0.01 & 0.02  &  0.10   & 0.49  & 2.83  &  13.71 \\[0.2cm]
		
		TSCSP             & $\alpha_{opt}$ & 0.94  & 0.94 & 0.94  & 0.94   & 0.94  &  0.94 \\
		& Iter       & 2     & 2    & 2     & 2  & 2       &   3 \\
		& CPU        & 0.01  & 0.02 & 0.10  &  0.50   & 2.97    &  16.78  \\[0.2cm]
		
		SCSP             & $\alpha_{opt}$    & 0.98    & 0.99    & 0.99    & 0.99    & 0.99    & 0.99  \\
		& Iter         & 3        & 3         &  3        &  3        & 4   &  5    \\
		& CPU          &  0.01  &   0.02  & 0.07    &  0.37   &  2.19   &  12.10  \\[0.2cm]
		
		
		PMHSS             & $\alpha_{opt}$   & 0.91    & 0.91    & 0.91    & 0.91     & 0.91   &  0.91 \\
		& Iter        & 20       & 20       & 20      & 20        & 20      & 20     \\
		& CPU         &  0.01    &  0.05   & 0.34   & 1.84   & 14.45  &   71.53 \\ [0.2cm] \hline

		
		BiCGSTAB     &     Iter    &  62.5  &  117.5  &  220.5  & 414.5  & 483.0   & $\dag$  \\
		& CPU        &    0.04     &  0.10      &  0.46     & 3.06    &  23.63    &  \\[0.2cm]

		BiCGSTAB-ILU     &     Iter    &  11.5  &  18.0  &  24.0  & 35.5  & 48.0   & 59.5  \\
		& CPU        &    0.03     &  0.04      &  0.11     & 0.57    & 5.36   & 24.74 \\[0.2cm]
		
		BiCGSTAB-TTSCSP        & $\alpha_{opt}$           & 0.37      & 0.49    & 0.58    & 0.63    & 0.65    & 0.65  \\
		& $\beta_{opt}$   & 1.00      & 1.00    & 1.00    & 1.00    & 1.00    & 1.00  \\
		& Iter     & 1.0   & 1.0    &  1.0  &  1.0   &  1.0   &  1.0    \\
		& CPU     &  0.02   &  0.04    & 0.14    & 0.72    & 4.52    & 22.03  \\[0.2cm]
		
		BiCGSTAB-TTSCSP         & Iter  & 1.0   & 1.0    &  1.0    &  1.0   &  1.0 &  1.0   \\
		$\hspace{0.7cm}\alpha=\beta=1$        & CPU   &  0.01    & 0.02     & 0.10    &  0.53   & 3.81  & 18.40 \\[0.2cm] \hline

	\end{tabular}
\end{table}


\begin{table}[!ht]
	\centering\caption{Numerical results of ITTSCSP, ITSCSP, ISCSP and IPMHSS  for Example \ref{Ex1} with $\tau=500h$. \label{Tbl2-b}}
	\bigskip
	\begin{tabular}{llllllllll} \hline
		Method            &  $m$   & 32    & 64   &128    &  256    & 512   & 1024   &    2048  &   \\ \hline \vspace{-0.3cm} \\ \vspace{0.0cm}
		
		ITTSCSP      & $\alpha_{opt}$   & 0.85   & 0.85   & 0.85    & 0.85   & 0.85   & 0.85  &  0.85 \\
		& $\beta_{opt}$    & 1.00  & 1.00  & 1.00   & 1.00  & 1.00   &  1.00 &  1.00 \\
		& Iter           & 2  &   2& 2  & 2 &  2  & 2& 3 \\
		& CPU           &  0.03 & 0.04 & 0.11 & 0.44  & 4.40 & 23.35 & 189.71 \\
		[0.2cm]
		
		ITSCSP             & $\alpha_{opt}$  & 0.94   & 0.94     & 0.94    & 0.94   & 0.94    & 0.94   & 0.94  \\
		& Iter        & 2   & 2   & 2   &  2  & 3  &  3   &  3 \\
		& CPU         &  0.03 &  0.05 & 0.10  & 0.43  & 7.70  & 45.96 & 233.25 \\[0.2cm]
		
		ISCSP             & $\alpha_{opt}$  & 0.99    & 0.99     & 0.99     & 0.99    & 0.99     & 0.99   & 0.99  \\
		& Iter                   & 3        & 3          & 3         &  4        & 4       &  5 & 5 \\
		& CPU                   & 0.03   &  0.04    &  0.09   & 0.49   & 6.07    & 40.13 & 204.43  \\
		
		IPMHSS          & $\alpha_{opt}$   & 0.91    & 0.91   & 0.91   & 0.91   & 0.91  & 0.91   &   \\
		& Iter    & 20       & 20       & 20      & 20        & 20       & 20      & $\ddagger$\\
		& CPU    & 0.06 &  0.25 &  1.20 & 6.69   &   69.52 &  389.97 &    \\  \hline
		
	\end{tabular}
\end{table}

\begin{example}\label {Ex1}\rm
See \cite{Bai2,Bai3} Consider the system of linear equations
\begin{equation}
[(K+\frac{3-\sqrt{3}}{\tau})+i\big(K+\frac{3+\sqrt{3}}{\tau}I)]x=b,
\end{equation}
where $\tau$ is the time step-size and $K$ is the five-point centered difference matrix approximating the negative Laplacian operator $L\equiv -\Delta$ with homogeneous Dirichlet boundary conditions, on a uniform mesh in the unit square $[0, 1] \times [0, 1]$ with the mesh-size $h=1/(m + 1)$. The matrix $K\in\mathbb{R }^{n \times n}$ possesses the tensor-product form $K=I\otimes V_m + V_m \otimes I$, with $V_m=h^{-2} \textrm{tridiag} (-1, 2,-1)\in \mathbb{R }^{m \times m}$. Hence, $K$ is an $n \times n$ block-tridiagonal matrix, with $n = m^2$. We take
$$
W=K+\frac{3-\sqrt{3}}{\tau}I, \quad \textrm{and}  \quad T=K+\frac{3+\sqrt{3}}{\tau}I
$$
and the right-hand side vector $b$ with its $j$th entry $b_j$ being given by
$$
b_j=\frac{(1-i)j}{\tau(j+1)^2},  \quad  j=1,2, \ldots ,n.
$$
 {In our tests, we take $\tau= h$ and $\tau=500h$.} Furthermore, we normalize coefficient matrix and right-hand side by multiplying both by $h^2$.
\end{example}


\begin{table}[!ht]
\centering
\caption{Numerical results of TTSCSP, TSCSP, SCSP and PMHSS for Example \ref{Ex2}.} \label{Tbl3}
\bigskip
\begin{tabular}{lllllllll} \hline
Method            &  $m$                & 32        & 64       &128      &  256     & 512        & 1024  &   \\ \hline \vspace{-0.3cm} \\ \vspace{0.0cm}

TTSCSP              & $\alpha_{opt}$       & 0.4       & 0.4      & 0.45    & 0.45     & 0.45     & 0.45  \\
                    & $\beta_{opt}$        & 0.1       & 0.1      & 0.1      & 0.1      & 0.1        & 0.1\\
                    & Iter                        & 10        & 9         & 8         &  8        & 8           & 8\\
                    & CPU                      & 0.01      & 0.03    & 0.20     & 0.90    & 6.27      &  28.74 \\[0.2cm]

TSCSP        & $\alpha_{opt}$ & 0.09      & 0.08      & 0.07     & 0.06     & 0.06      &  0.06  \\
                  & Iter                   & 22         & 24         & 23        &  23      & 21          & 20\\
                  & CPU                  & 0.02       & 0.06      & 0.45      & 2.04     & 13.90      & 61.47  \\[0.2cm]

SCSP        & $\alpha_{opt}$    & 1.35      & 1.37      & 1.42     & 1.43     &  1.47     & 1.48   \\
                  & Iter                   & 38         & 38         & 36        &  35      &   33        &  32\\
                  & CPU                  & 0.02       & 0.06      & 0.33     & 1.51     &10.69       & 51.08  \\[0.2cm]


PMHSS            & $\alpha_{opt}$  & 0.98    & 0.93      & 1.1      & 0.97      & 0.97    & 1.0\\
                       & Iter                   & 37       & 38        & 38       & 38         & 38       & 38\\
                      & CPU                    & 0.03     & 0.08      & 0.62      & 3.10      & 23.24   & 116.50\\[0.2cm]  \hline


BiCGSTAB   & Iter   & 41.5     &  83.5   &   163.5   & 368.5  &  $ \hspace{0.2cm} \dag$ &  $ \hspace{0.2cm} \dag$ \\
                                     & CPU                          &  0.04     &  0.08      &  0.34   & 2.79    &       &  \\[0.2cm]

BiCGSTAB-ILU          & Iter     & 0.5     &  0.5    &   0.5   & 0.5  &  0.5 &  0.5 \\
                                     & CPU                          &  0.02     &  0.02   &  0.03   & 0.06    &  0.26     & 0.91 \\[0.2cm]

BiCGSTAB-TTSCSP        & $\alpha_{opt}$            & 0.40      & 0.40    & 0.45    & 0.45    & 0.45    & 0.45  \\
                                    & $\beta_{opt}$             & 0.10      & 0.10    & 0.10    & 0.10    & 0.10    & 0.10  \\
                                    & Iter                             & 3.5   & 3.5    &  3.5     &  3.0     &  3.0       &  2.5    \\
                                    & CPU                            & 0.03      & 0.06     & 0.22    & 1.03   & 7.17    & 22.30 \\[0.2cm]

BiCGSTAB-TTSCSP                                & Iter     & 3.5   & 3.5       &  3.5  &  3.0        &  2.5     &   2.5   \\
   $\hspace{0.7cm}\alpha=\beta=1$     & CPU    & 0.01      & 0.03     & 0.18    &  0.88  & 5.78  &  28.71 \\[0.2cm] \hline

 \end{tabular}
\end{table}

\renewcommand{\thefootnote}{\fnsymbol{footnote}}
\begin{table}[!ht]
\centering
\caption{Numerical results of ITTSCSP, ITSCSP, ISCSP and IPMHSS  for Example \ref{Ex2}.} \label{Tbl4}
\bigskip
\begin{tabular}{llllllllll} \hline
Method            &  $m$            & 32        & 64       &128      &  256     & 512        & 1024  &  2048 \\ \hline \vspace{-0.3cm} \\ \vspace{0.0cm}

ITTSCSP       & $\alpha_{opt}$    & 0.4       & 0.4      & 0.42     & 0.4      & 0.4        & 0.4      &  0.4 \\
                  & $\beta_{opt}$      & 0.12      & 0.09    & 0.09     & 0.09    & 0.09      & 0.09     &  0.09 \\
                  & Iter                      &  9         & 9         &  8         &  8        & 8           & 8         &  8\\
                  & CPU                     & 0.05      & 0.10    & 0.31      & 1.56    & 12.21    & 64.88   &  397.67\\[0.2cm]

ITSCSP        & $\alpha_{opt}$& 0.1        & 0.08      & 0.07     & 0.07     & 0.07       & 0.06     & 0.06\\
                  & Iter                   & 23         & 27         & 25        &  24       & 24          & 23        & 22\\
                  & CPU                  & 0.09       & 0.25      & 0.89     & 5.63     & 48.31     & 242.39  & 1335.77 \\[0.2cm]

ISCSP        & $\alpha_{opt}$ & 1.35      & 1.37      & 1.39     & 1.43     & 1.45       &  1.46      & 1.47\\
                  & Iter                 & 38         &  38        & 37        &  35       & 34          &    33       & 32\\
                  & CPU                 & 0.07      & 0.16      & 0.54     &  2.82    & 21.81     &   121.12  &  633.45\\[0.2cm]

IPMHSS         & $\alpha_{opt}$  & 0.78       & 0.82    & 0.75   & 0.77     & 0.82       & 0.94        & \\
               & Iter            & 36         & 37      & 38     &  38      & 38         & 39          &  $\hspace{0.2cm} \ddag$ \\
               & CPU             & 0.13       & 0.41    &1.65    & 9.91     & 75.90      & 497.06      & \\
               & \verb"droptol"  & 1e-2       & 5e-3    & 1e-3   & 5e-4     & 1e-5       & 1e-5  &  \\  \hline

 \end{tabular}
\end{table}


\begin{example}\label {Ex2}\rm
(See \cite{Bai2,Bai3}) Consider the system of linear  equations $\eqref{Eq1} $ as following
$$
\big[ (-\omega ^2 M+K) +i(\omega C_V +C_ H) \big] =b,
$$
where $M$ and $K$ are the inertia and the stiffness matrices, $C_V$ and $C_H$ are the viscous and the hysteretic damping matrices, respectively, and $ \omega$ is the driving circular frequency. We take $C_H = \mu K$ with $\mu$ a damping coefficient, $M = I, C_V = 10I$, and $K$ the five-point centered difference matrix approximating the negative Laplacian operator with homogeneous Dirichlet boundary conditions, on a uniform mesh in the unit square $[0, 1] \times [0, 1]$ with the mesh-size $h =1/(m+1)$. The matrix $K \in \mathbb{R }^{n \times n}$ possesses the tensor-product form $K = I \otimes V_m+V_m \otimes I$, with $V_m =h^{-2} \textrm{tridiag}(-1, 2,-1) \in\mathbb{R }^{m \times m} $. Hence, $K$ is an $n \times n$ block-tridiagonal matrix, with $n = m^2$. In addition, we set $\mu = 0.02$, $\omega = \pi$, and the right-hand side vector b to be $b = (1 + i)A\textbf{1}$, with \textbf{1} being the vector of all entries equal to 1. As before, we normalize the system by multiplying both sides through by $h^2$.
\end{example}

\begin{table}[!ht]
\centering
\caption{Numerical results of TTSCSP, TSCSP, SCSP and PMHSS  for Example \ref{Ex3}.} \label{Tbl5}
\bigskip
\begin{tabular}{lllllllll} \hline
Method            &  $m$   & 32    & 64    &128    &  256  & 512     & 1024  &       \\ \hline \vspace{-0.3cm} \\ \vspace{0.0cm}

TTSCSP           & $\alpha_{opt}$    & 0.72     & 0.48   & 0.32    & 0.23    & 0.16      & 0.12 \\
                       & $\beta_{opt}$     & 0.2      & 0.2     & 0.2      & 0.2      & 0.2         & 0.2 \\
                      & Iter                      & 6         & 8        & 10       &  12     &   14         & 15 \\
                      & CPU                     & 0.01    & 0.05   & 0.20    &  1.82   &   15.05    & 86.96 \\[0.2cm]

TSCSP             & $\alpha_{opt}$    & 0.23    & 0.23     & 0.23      & 0.23    & 0.16   & 0.11\\
                      & Iter                      & 13       & 13        & 13         &  13      & 16      &  23\\
                      & CPU                     & 0.02    & 0.07     & 0.36      & 1.91    & 17.16  & 128.41\\[0.2cm]

SCSP             & $\alpha_{opt}$     & 1.92    & 1.44     & 1.15      & 1.02    & 0.96      &  0.93\\
                      & Iter                     & 15       & 25        & 40         &  59      & 78         &  94 \\
                      & CPU                     & 0.02    & 0.06     & 0.50      &  3.69    & 39.21    &   235.17 \\[0.2cm]


PMHSS             & $\alpha_{opt}$    & 0.42     & 0.57    & 0.78     & 0.73     & 0.76      & 0.81 \\
                       & Iter                      & 30        & 30       & 30        & 30       & 31         & 32\\
                       & CPU                     & 0.03      & 0.11    & 0.73     & 3.83    & 30.55    & 169.04\\[0.2cm] \hline


BiCGSTAB                      & Iter                           &  51.5     &  99.5    &   196.5   & 378  & $ \hspace{0.2cm} \dag$  &  $\hspace{0.2cm} \dag$ \\
                                     & CPU                         &  0.04    &  0.09      &  0.42      &      2.77        &              &  \\[0.2cm]

BiCGSTAB-ILU        & Iter   &  0.5     &  0.5    &   0.5   & 0.5  & 0.5  &  0.5 \\
                   & CPU  &  0.02    &  0.02      &  0.03      &   0.06    & 0.21    & 0.92 \\[0.2cm]

BiCGSTAB-TTSCSP        & $\alpha_{opt}$            & 0.72      & 0.48    & 0.32    & 0.23    & 0.16    & 0.12  \\
                                    & $\beta_{opt}$             & 0.20      & 0.20    & 0.20    & 0.20    & 0.20    & 0.20  \\
                                    & Iter                             & 3.0         & 3.5      &  4.0       &  4.5     &  5.0      &  5.5    \\
                                    & CPU                            & 0.03      & 0.06    & 0.33    & 1.82    & 15.37  & 84.19  \\[0.2cm]

BiCGSTAB-TTSCSP                         & Iter     & 3.5        & 3.5      &  3.5      &  3.5     &  3.5       &  3.5    \\
 $\hspace{1cm}\alpha=\beta=1$   & CPU    & 0.01      & 0.04     & 0.27    & 1.45  & 11.38  & 57.97 \\[0.2cm] \hline

 \end{tabular}
\end{table}
\begin{table}[!ht]
\centering
\caption{Numerical results of ITTSCSP, ITSCSP, ISCSP and IPMHSS  for Example \ref{Ex3}.} \label{Tbl6}
\bigskip
\begin{tabular}{lllllllll} \hline
Method            &  $m$          & 32    & 64    &128    &  256  & 512     & 1024  &   2048    \\ \hline \vspace{-0.3cm} \\ \vspace{0.0cm}

ITTSCSP          & $\alpha_{opt}$    & 1.10    & 0.53     & 0.35    & 0.22     & 0.16    & 0.12     & 0.10 \\
                       & $\beta_{opt}$     & 0.16    & 0.16     & 0.16    & 0.16     & 0.16    &  0.15    &  0.15\\
                      & Iter                      & 6         & 8         & 11       &  14       & 17       & 19        &  21\\
                      & CPU                      & 0.04    & 0.09    & 0.35    & 2.42     & 23.76   & 128.66 &  847.67  \\[0.2cm]

ITSCSP             & $\alpha_{opt}$   & 0.19    & 0.18     & 0.17      & 0.16    & 0.15     & 0.11      &  0.08\\
                      & Iter                      & 16       & 16        & 17         &  17      & 17        & 24         &  34    \\
                      & CPU                     & 0.07    & 0.15     & 0.54      & 2.84     & 24.91   & 223.65   & 1270.78  \\[0.2cm]

ISCSP             & $\alpha_{opt}$  & 1.9     & 1.38     & 1.16      & 1.05    &  1.00     &  0.98     &  0.96 \\
                      & Iter                   & 15      & 25        & 41        &  70      &   108     &   150     &   182  \\
                      & CPU                   & 0.04   & 0.12     & 0.65     & 5.49     & 71.12    & 495.23  &   3484.17 \\[0.2cm]


IPMHSS             & $\alpha_{opt}$   & 0.33     & 0.42    & 0.54     & 0.65     & 0.71   & 0.9        &  1.2\\
                       & Iter                      & 30        & 30       & 30        & 31       & 32       & 33         &  38 \\
                       & CPU                     & 0.1       & 0.26    & 1.05     & 5.78     & 51.62   & 270.88  & 1596.2\\  \hline

 \end{tabular}
\end{table}
\begin{example}\label {Ex3}\rm
(See \cite{Bai2,Bai3}) Consider the system of linear equations $\eqref{Eq1}$ as following
$$
T=I\otimes V + V \otimes I  \quad   \textrm{and}  \quad  W=10(I \otimes V_c + V_c \otimes I) +9(e_1 e^T_m + e_m (e^T_m) \otimes I,
$$
where $V = \textrm{tridiag}(-1, 2,-1) \in \mathbb{R}^{m \times m}$, $V_c = V - e_1 e^T_m - e_m e^T_1 \in \mathbb{R}^{m \times m}$ and $e_1$ and $e_m$ are the first and last unit vectors in $R^m$, respectively. We take the right-hand side vector $b$ to be $b = (1 + i)A\textbf{1}$, with \textbf{1} being the vector of all entries equal to 1.  Here $T$ and $W$ correspond to the five-point centered difference matrices approximating the negative Laplacian operator with homogeneous Dirichlet boundary conditions and periodic boundary conditions, respectively, on a uniform mesh in the unit square $[0, 1]\in [0, 1]$ with the mesh-size
$h = 1/(m+ 1)$.
\end{example}


Numerical results for Examples  \ref{Ex1}-\ref{Ex3} are listed in Tables \ref{Tbl1}-\ref{Tbl6}. In  Tables \ref{Tbl1}, \ref{Tbl1-b},  \ref{Tbl3} and \ref{Tbl5} the numerical results of the exact version of the iteration methods are presented and those of the inexact versions are given in Tables \ref{Tbl2}, \ref{Tbl2-b}, \ref{Tbl4} and \ref{Tbl6}. For the PMHSS, the SCSP, the TSCSP and the TTSCSP  iteration methods, the optimal value of $\alpha$ ($\alpha_{opt}$) were found experimentally and are the ones resulting in the least numbers of iterations.  In Table \ref{Tbl4}, for $m\geqslant 64$,  the \verb"ict" function of \textsc{Matlab} encounters a nonpositive pivot during the computation of the inexact Cholesky factorization. Therefore, we have used a smaller value of the dropping tolerance (\verb"droptol") which have been presented in the table.

 {As the numerical results show for all the examples, the TTSCSP iteration method often  outperforms the other methods in terms of both the number of iterations and the CPU time.}

From Tables \ref{Tbl1}, \ref{Tbl2}, \ref{Tbl1-b} and \ref{Tbl2-b} we see that, for Example \ref{Ex1}, the iteration counts with TTSCSP and ITTSCSP are  the same and with problem size remain almost constant. From the CPU time of view, we observe that the ITTSCSP iteration method is superior to the TTSCSP iteration method for large problems. On the other hand, the optimal values of the parameters remain constant for both of the TTSCSP and ITTSCSP iteration methods.
Almost all of these comments can be posed for Example \ref{Ex2}.

Numerical results for Example \ref{Ex3} show that the iteration counts with the TTSCSP and the ITTSCSP iterations growth moderately with problem size.
Also, this table show the optimal value of the parameter $\alpha$ remains almost constant with problem size for both of the TTSCSP and the ITTSCSP methods, whereas the optimal value of the parameter $\alpha$ decreases moderately.

 {From Tables \ref{Tbl1}, \ref{Tbl1-b}, \ref{Tbl3} and \ref{Tbl5} we see that the TTSCSP preconditioner, both with the optimal values of the parameters and $\alpha=\beta=1$, is very effective in reducing the number of iterations of the BiCGSTAB iteration method as well as the CPU time. Moreover, there is not any significant difference between the numerical results of the TTSCSP preconditioner with the optimal values of the parameters and $\alpha=\beta=1$. In Examples \ref{Ex2} and \ref{Ex3}, for $m=1024$,  we observe that the BiCGSTAB method does not converge in 500 iterations, whereas the preconditioned BiCGSTAB method with the TTSCSP preconditioner converges only in 3.5 and 5.5 iterations, respectively. In Example \ref{Ex1} we see that the iteration counts of the  BiCGSTAB with the TTSCSP preconditioner remain constant with the problem size, those of Example \ref{Ex2} decreases, and those of Example \ref{Ex3} increases by 0.5.}

 {From Tables \ref{Tbl1} and \ref{Tbl1-b} we see that the iteration counts  of the BiCGSTAB method with the TTSCSP preconditioner is always less than that of with the ILU preconditioner for both of the parameters $\tau=h$ and $\tau=500h$. However, the CPU time of the BiCGSTAB with the TTSCSP preconditioner is always less than that of with the ILU preconditioner when $\tau=500h$, and the result is opposite for $\tau=h$. From Tables \ref{Tbl3} and \ref{Tbl5} we see that the ILU preconditioner outperforms the TTSCSP preconditioner for Examples \ref{Ex2} and \ref{Ex3} from both the CPU time point of view and the number of iterations. Nevertheless, the TTSCSP preconditioner has a main advantages over the ILU preconditioner. In the implementation of
the TTSCSP preconditioner two systems with the coefficient matrices $\alpha W+T$ and $W+\beta T$ should be solved. If these systems are solved inexactly by using the CG method, then there is not  any additional matrix to store, however, in the ILU preconditioning the ILU factors  of the matrix $A$ should be stored.}

 {Using the TTSCSP and the ILU preconditioners for the BiCGSTAB iteration method result in faster solution times than using TTSCSP as a stationary method for some problems. However, the BiCGSTAB method needs additional operations such as inner products. Inner products require global communication on parallel computers and they are a parallel bottleneck on  current multicore architectures. Therefore, it may be better to apply the TTSCSP method as a stationary method for solving the system in some cases.}

\section{Conclusion}\label{Sec5}
We have established and analyzed a two-parameter TSCSP iteration (TTSCSP) method for solving an important class of complex symmetric system of  linear equations $(W+iT)u=b$, where $W$ is symmetric positive definite and $T$ is symmetric positive semidefinite. Sufficient conditions for the convergence of the method have also been presented. An upper bound for the spectral radius of the iteration matrix along with the parameters which minimize this bound have been given. We have compared the numerical results of the TTSCSP iteration method with those of the SCSP, the TSCSP and PMHSS iteration methods. Numerical results show that the TTSCSP method is superior to the other methods in terms of both the iteration counts and the CPU time.
Numerical comparisons of the inexact TTSCSP (ITTSCSP) with ISCSP, ITSCSP and IPMHSS methods have also been presented which show the superiority of the ITTSCSP to the other methods. Numerical results show that the BiCGSTAB method in conjunction with the TTSCSP preconditioner is very effective for solving  $(W+iT)u=b$.

\section*{Acknowledgements}

 {The work of Davod Khojasteh Salkuyeh is partially supported by University of Guilan. The authors would like to thank Prof. M. Benzi and anonymous referees for their valuable  comments and suggestions which greatly improved the  quality of the paper.}

\end{document}